\documentclass[11pt]{amsart}
\usepackage{amsmath}
\usepackage{amsfonts}
\usepackage{amsthm}
\usepackage{amssymb}
\usepackage{amscd}
\usepackage[all]{xy}
\usepackage{enumerate}
\usepackage{color}
\usepackage{bm}
\usepackage{amstext}



\theoremstyle{plain}
\newtheorem{thm}{Theorem}[subsection]

\newtheorem{prop}[thm]{Proposition}

\newtheorem{cor}[thm]{Corollary}
\newtheorem{lem}[thm]{Lemma}

\theoremstyle{definition}
\newtheorem{defn}[thm]{Definition}

\newtheorem*{ackn}{Acknowledgements}

\newtheorem{rem}[thm]{Remark}

\newcommand{\sE}{\mathcal{E}}
\newcommand{\sF}{\mathcal{F}}

\newcommand{\sO}{\mathcal{O}}

\newcommand{\sQ}{\mathcal{Q}}

\newcommand{\sS}{\mathcal{S}}
\newcommand{\sT}{\mathcal{T}}

\newcommand{\mA}{\mathbb{A}}
\newcommand{\mC}{\mathbb{C}}

\newcommand{\mG}{\mathbb{G}}

\newcommand{\mP}{\mathbb{P}}

\newcommand{\SL}{\mathrm{SL}\,}

\newcommand{\Ga}{{\mathbb{G}_a}}

\numberwithin{equation}{section}

\newenvironment{fcaption}{\begin{list}{}{
\setlength{\leftmargin}{35pt}
\setlength{\rightmargin}{35pt}
\setlength{\labelsep}{5pt}
}}{\end{list}}

\author{Pietro Corvaja Francesco Zucconi}

\address{\newline D.M.I.F. \\
Univerit\`a degli studi di Udine\\
Udine, 33100, Italy.\newline
\texttt{pietro.corvaja@uniud.it\newline
francesco.zucconi@uniud.it}}

\begin{document}

\begin{center}
\textbf{\Large
Bitangents to a quartic surface \\
and infinitesimal deformations \\}
\par
\end{center}
{\Large \par}

$\;$

\begin{center}
Pietro Corvaja, Francesco Zucconi 
\par\end{center}

\vspace{5pt}

$\;$

\begin{fcaption} {\small  \item 
Abstract. 
We prove that the Hilbert scheme which parametrises bitangent lines to a general quartic surface is a smooth regular surface with no rational curves and with very ample canonical divisor. We also prove that it is a counterexample to infinitesimal Torelli and its infinitesimal deformation space is of dimension $20$.
}\end{fcaption}

\vspace{0.5cm}




\section{Introduction} 

We work over $\mC$, the complex number field. 

In \cite{To} Ruggero Torelli introduced the concept of the Jacobi variety of a curve. His fundamental result is that two curves are isomorphic if and only if their Jacobian varieties are isomorphic as principally polarized Abelian varieties.
In \cite{G} Phillip Griffiths extended these investigations considering period maps, associating to a variety its Hodge structure, and asked whether for a simply connected surface of general type local Torelli holds, i.e., the local period map is an embedding.
This property is valid if the infinitesimal period map is injective. Griffiths showed that to compute the differential of the period map is equivalent to compute the cup product map between certain cohomological vector spaces. Since then a huge work has been devoted to the study of the differential of the period map; c.f. see: \cite{CMP}.
Here we can only recall that any infinitesimal deformation of a smooth variety $S$ is a flat morphism 
$\sS\to{\rm{Spec}}   \frac{\mathbb C[\epsilon]} {\epsilon^2}$, whose fiber over $0$ is $S$, and that it comes equipped with the extension class $\xi\in H^1(S,T_S)$ of the following exact sequence:
$$0\to\sO_S\to \sE\to\Omega^1_S\to 0$$
\noindent  where $T_S$ is the tangent sheaf of $S$, $\Omega^1_S$ is the cotangent sheaf of $S$ and $\sE$ is the cotangent sheaf of $\sS$ restricted to $S$; see: cf \cite[Theorem 4.50]{Vo}, \cite[Chap. 9, \S {} 9.1.2]{Vo}. 
The $i$-th exterior products in the above sequence yields $n$ short exact sequences 
$$0\to\Omega^{i-1}_S\to\bigwedge^{i} \sE\to\Omega^{i}_{S}\to 0$$
\noindent and $n$ co-boundary homomorphisms 
$\partial^{i}_{\xi} \colon H^0(S,\Omega^{i}_{S})\to H^1(S,\Omega^{i-1}_{S} )$, where $i=1,..., n$ and $n$ is the dimension of $S$. These co-boundary homomorphisms are computed via cup product with the class $\xi$. Essentially the differential of the Griffiths' period map is given by these cup products and to write that $S$ has the {\it{infinitesimal Torelli property}}, (iTp in the sequel) means that if the homomorphisms $\partial^{i}_{\xi}$, $i=1,...,n$ are all trivial then $\xi$ is trivial too. 

The literature on iTp is vast. We recall that if $n=1$ iTp
holds iff $g(C) = 1,2$ or iff
$g(C) \geq 3$ and $C$ is not hyperelliptic {\it cf.} \cite{To}, \cite{A}, 
\cite{We}, \cite{OS}; see also \cite{RZ4}. In higher dimensions, thanks to the spectacular interpretation given in \cite{G} of the infinitesimal variation of Hodge structure via natural homomorphisms of the Jacobian ring of a hypersurface, many authors were able to give an algebraic interpretations of iTp and to obtain many important results; c.f. see: \cite{CMP}. In \cite{Fl} it is shown that iTp holds for any smooth complete intersection inside $\mP^n$ with the only exception of hypersurfaces of degree $3$ in $\mP^3$ and intersections of two quadrics of even dimension where iTp does not hold.
Moreover also in other contexts,  as in the case of complete intersections inside homogeneous varieties, we find affermative answers to iTp;
see: \cite{Ko1}, \cite{Ko2}, \cite{Ko3},\cite{Ko4}, \cite{RZ1}, \cite{RZ2}, \cite{RZ3}. 

It is true that in some cases, as the one constituted by surfaces of general type with $p_g=q=0$ (where $p_g={\rm{dim}}_{\mathbb C} H^0(S,\Omega^2_S)$, $q={\rm{dim}}_{\mathbb C} H^0(S,\Omega^1_S)$), the hypothesis of iTp obviously holds while it does not hold, see \cite{Ca}. Here we have to mention that there are classes of counterexamples to the global Totelli claim, that is, with non injective period map. They are obtained by a deep study of
surfaces with $p_g = 1$, $q = 0$ and $1\leq K^2\leq 8$, see \cite{Td1}, \cite{Td2} and \cite{Ch}. In the case where $p_g=1$, $q=0$ and $K^2=2$ there exist also simply connected surfaces which are counterexample to the global Torelli property. However all these surfaces are rather special; in particular the canonical ring exhibits unusual properties or they contain special cycles; see: \cite{Td2}, \cite{Ch}.
There are counterexamples to iTp where the canonical sheaf is quasi very ample,
  i.e., the canonical map is a birational morphism and a local embedding on the complement of a finite set; see: \cite{BC}.  In the case of irregular surfaces there are counterexamples with very ample canonical sheaf: \cite{GZ}. There are also some counterexamples for varieties contained inside weighted projective spaces, see \cite{FRZ}. Nevertheless in each of these classes of counterexamples, as well as in the above mentioned case of hyperelliptic curves, one can obtain the negative answer to iTp by explicit algebraic computations based on special features of the (anti)canonical ring or by the existence of special kind of fibrations on the variety. In this sense one could have been led to expect that for a regular variety $S$ of general type (that is a variety of general type with no $(1,0)$-holomorphic forms)  iTp holds under standard geometrical assumptions as the very ampleness of the canonical sheaf, and the non existence of rational subvarieties inside $S$.

The counterexamples built in this work give a negative answer also for these cases. They are provided by surfaces classifying  the bitangents lines to general quartic surfaces.  Actually, the original motivation for this work was indeed the study of such surfaces, including the determination of all their numerical invariants, and their relations to some arithmetical problems concerning the geometry of quartic surfaces; see \cite{CZ1}, \cite{CZ2}. When our work was initiated, we came across the beautiful PhD-thesis of G.E. Welters, which contained the answer to many of our questions; the present work can be viewed in part as a gloss on Welters' paper \cite{W}.

\medskip
 Let $X\subset \mP^3$ be a general quartic surface, where "general" in this paper only means  that $X$ is smooth and does not contain any line of $\mP^3$. In \cite{Cl} the (unique) double cover $Q$ of $\mP^3$ branched on $X$ is called double solid. 
The works of Tihomirov \cite{T} and Welters \cite{W} contain  a detailed study of the Hilbert scheme $S_X$ of  `lines' of $Q$, i.e. curves having degree one with respect to a certain ample (but not very ample) divisor; we observe that these so called lines do not become actual lines under any projective embedding of the threefold.
In the same papers are laid the foundations to study double solids and the geometry associated to their intermediate Jacobians. We recall that Bombieri and Swinnerton-Dyer, in their seminal paper \cite{BSD}, used intermediate Jacobians precisely to study the variety of lines on a threefold; since then the use of intermediate Jacobians became a common tool to investigate rationality questions.
 In the quoted papers \cite{T} and \cite{W}  it is shown how to construct an \'etale double cover $f\colon S_X\to S$, where $S$ is the surface in the Grassmannian $\mathbb G(2,4)$ which parametrises bitangent lines to $X$. By a result of Welters it follows that the differential of the Albanese morphism $a_X\colon S_X\to{\rm{Alb}}(S_X)$ is injective; see Theorem \ref{collino} below. This implies that there are no rational curves on $S$; see Theorem \ref{norational}. In \cite {W}  a detailed study of the Gauss map associated to $a_X\colon S_X\to{\rm{Alb}}(S_X)$ is carried out. By this study we can explicitly clarify that the canonical morphism of $S$ is actually an embedding; see Theorem \ref{veryampleness}. We show that the tangent bundle sequence, see the sequence (\ref{simsigma}), provides a non trivial class $\xi$ such that $\partial^{2}_{\xi}=0$. In particular this means that iTp does not hold for $S$; we remark that since $S$ is a regular surface then $\partial^{1}_{\xi}$ is obviously trivial. 
The paper contains a study of the canonical image of $S$. In particular to show that $S$ is not $2$-normal we prove that the co-kernel of the standard multiplicative map:
$$
\mu\colon H^0(S,\omega_S)\otimes_\mathbb C H^0(S,\omega_S)\to H^0(S,\omega^{\otimes 2}_S)
$$
has dimension at least equal to $h^1(S,T_S)$; see Theorem \ref{nondue}. We end the paper showing that $h^1(S,T_S)=20$; see Theorem \ref{domandareferee}. We think that the fact that the dimension of the infinitesimal deformation space of a smooth quartic surface is the same as the one of the surface of its bitangent lines deserves to be studied.

Finally, our main result reads:

\noindent
{\bf{Main Theorem}}  {\it{Let $X$ be a quartic surface which contains no lines. Then the surface $S$ which parametrises its bitangent lines is a smooth surface of general type with very ample canonical sheaf. It contains no rational curves, and the infinitesimal Torelli property does not hold for it. Moreover the canonical model is not $2$-normal and the dimension of the infinitesimal deformations space of $S$ is $20$. }}

\medskip

See Theorem \ref{teoremaA}, Theorem \ref{nondue} and Theorem \ref{domandareferee}.
 In the rest of this introduction we will give an account of infinitesimal Torelli property and we will describe our construction.

\smallskip

Our study on $S$ gives not only a negative answer to some questions concerning the infinitesimal Torelli problem but also it shows that (iTp) does not hold for any of codimension $\geq 2$ subvariety of a Grassmannian, under standard geometrical assumptions. This contrasts with the well known fact that (iTp) holds  for complete intersections in certain homogeneous K\"ahler manifolds, under the same standard assumptions; see: \cite{Ko1} \cite{Ko2}, \cite{Ko3}, \cite{Ko4}, and indeed we know that it holds also for hypersurfaces in log parallelizable varieties; see \cite{R}, and conjecturally it holds for complete intersections in log parallelizable varieties.

Finally $2$-normality appears to be a crucial hypothesis to show that iTp holds in the interesting case given by a regular surface with very ample canonical bundle, no rational curves in it, and $2$-normal canonical image; see: \cite{Re}. On the other hand the $2$-normality assumption restricts a lot the range of the claim, as we can see by our study, where the obstruction to $2$-normality is essentially given by the condition $h^1(S,T_S)\neq 0$.

\section{Notation}
\label{notazioni}

We follow Grothendieck's notation; that is if $\sE$ is a vector bundle on a variety $Z$, we denote by $\mP(\sE)$ the projectivization of the dual bundle. In particular if $T_{\sE}$ is the tautological line bundle on $\mP(\sE)$ and $\rho\colon \mP(\sE)\to Z$ is the natural projection then $\rho_{\star}\sO_{\mP(\sE)}(T_{\sE})=\sE$.
\begin{enumerate}
\item $V$ is a vector space such that ${\rm{dim}}_{\mathbb C}V=4$. 
\item $\mP^3$ is the projective space of the one-dimensional quotients of the $4$ dimensional vector space $V$.
\item $V^{\vee}:={\rm{Hom}}(V,\mathbb C)$.
\item $\mathbb G$ denotes the Grassmannian variety of lines in $\mP^3$.
\item $\sF^{\vee}$ is the dual vector bundle of a vector bundle $\sF$.
\item ${\rm{Sym}}^j\sF$ is the $j$-symmetric product of the vector bundle $\sF$.
\item $X$ denotes a smooth surface of degree $4$ inside $\mP^3$. 
\item In this paper {\it{$X$ is general}} means that it contains no line  in it.
\item $S\subset\mathbb G$ is the surfaces of bitangent lines to $X$.
\item $\sO_S(K_S)$, $\omega_S$, $\Omega^2_S$ are three ways to denote the canonical sheaf of $S$.
\item $Q$ is the double cover of $\mP^3$ branched over $X$.
\item $S_X$ is the Hilbert scheme of lines of $Q$.
\end{enumerate}

\section{Some vector bundles on the Grassmannian} In this section we follow the notation of \cite{AM}.
The universal exact sequence on the Grassmannian ${\mathbb G}$: 
\begin{equation}\label{univesattasequenza}
0\to \sS^\vee\stackrel{\phi}{\to} V\otimes \sO_{\mathbb G}\stackrel{\psi}{\to} \sQ\to 0
\end{equation}
defines two rank-$2$ vector bundles on $\mathbb G$: $\sQ$, $\sS$.


We consider $\mathbb G$ embedded inside $\mP^5$ by the Pl\"ucker embedding and we denote by $H_{\mathbb G}$ the corresponding hyperplane section.

\begin{lem}\label{ARRONDOondo} It holds: 
\begin{enumerate}
\item $\sO_{\mathbb G}(1)=\bigwedge^2\sQ=\bigwedge^2\sS$;
\item $({\rm{Sym}}^j (\sQ))^{\vee}=({\rm{Sym}}^j (\sQ))\otimes_{\sO_{\mathbb G}}\sO_{\mathbb G}(-jH_{\mathbb G})$.
\end{enumerate}
In particular $\bigwedge^2\sS^\vee=\sO_{\mathbb G}(-1)$.
\end{lem}
\begin{proof}Easy. See c.f. \cite{AM}.
\end{proof}

We will need to compute some cohomology groups on a surface contained in $\mathbb G$. We will use the following standard exact sequences on $\mathbb G$; see c.f. \cite[page 1100]{AM}:
\begin{equation}\label{laduesimmmm}
0\to\sO_{\mathbb G}(-1)\to \bigwedge^2 V\otimes_{\sO_\mathbb G}\sO_{\mathbb G}\to V\otimes_{\sO_\mathbb G}\sQ\to {\rm{Sym}}^2(\sQ)\to 0 
\end{equation}

\begin{equation}\label{latreeeesimmmm}
0\to\sQ(1)\to \sQ\otimes_{\sO_{\mathbb G}} {\rm{Sym}}^2 (\sQ)\to {\rm{Sym}}^3(\sQ)\to 0
\end{equation}
where for every $m\in\mathbb Z$, for every $j\geq 1$ and for every vector bundle $\sF$ on $\mathbb G$ we set
${\rm{Sym}}^j (\sF)(m):=({\rm{Sym}}^j (\sF))\otimes_{\sO_{\mathbb G}}\sO_{\mathbb G}(mH_{\mathbb G})$. 

\begin{lem}\label{detdelsim2}
$ \bigwedge^3{\rm{Sym}}^2(\sQ)=\sO_\mathbb G(3)$.
\end{lem}
\begin{proof} Since ${\rm{Sym}}^2(\sQ)\oplus\bigwedge^2 \sQ=\sQ\otimes_{\sO_\mathbb G}\sQ$
then $$\bigwedge^3{\rm{Sym}}^2(\sQ)\otimes_{\sO_\mathbb G}  \bigwedge^2 \sQ ={\rm{det}}(\sQ\otimes_{\sO_\mathbb G}\sQ)=({\rm{det}}\sQ)^{\otimes 4}.$$ Hence the claim follow by Lemma \ref{ARRONDOondo} $(1)$.
\end{proof}

\begin{cor}\label{wedge2sim2}
$\bigwedge^2 {\rm{Sym}}^2(\sQ)\simeq {\rm{Sym}}^2(\sQ) \otimes_{\sO_\mathbb G} \sO_\mathbb G(1)$
\end{cor}
\begin{proof} By the natural pairing and by Lemma \ref{detdelsim2}
$$
(\bigwedge^2 {\rm{Sym}}^2(\sQ))\otimes_{\sO_\mathbb G} {\rm{Sym}}^2(\sQ)\to\bigwedge^3 {\rm{Sym}}^2(\sQ)=\sO_\mathbb G(3).
$$
 It follows:
$$(\bigwedge^2 {\rm{Sym}}^2(\sQ))^\vee={\rm{Sym}}^2(\sQ)\otimes_{\sO_\mathbb G} \sO_\mathbb G(-3)$$
that is $\bigwedge^2 {\rm{Sym}}^2(\sQ)={\rm{Sym}}^2(\sQ^\vee)\otimes_{\sO_\mathbb G} \sO_\mathbb G(3)$. By Proposition \ref{ARRONDOondo} $(2)$ we then have:
$$\bigwedge^2 {\rm{Sym}}^2(\sQ))=({\rm{Sym}}^2 (\sQ))\otimes_{\sO_{\mathbb G}}\sO_{\mathbb G}(-2)\otimes_{\sO_\mathbb G} \sO_\mathbb G(3)={\rm{Sym}}^2(\sQ)\otimes_{\sO_\mathbb G} \sO_\mathbb G(1).$$
\end{proof}

\section{The surface of bitangent lines}

\subsection{Bitangent lines to a general quartic surface} We recall briefly some results taken from \cite{T} and \cite{W}. Let $X\subset \mP^3$ be a smooth quartic surface such that there are no lines contained inside $X$ (i.e. a general quartic surface). We denote by $F(x_0,x_1,x_2,x_3)\in\mathbb C[x_0,x_1,x_2,x_3]$ an homogeneous polynomial of degree $4$ such that  
$$X:=\{P\in\mP^3\mid F(P)=0\}.$$
\noindent

\begin{defn}\label{definizionedibitangente} A line $l\subset\mP^3$ is a {\it{bitangent line to}} $X$ if it is tangent to $X$ at each point of $l\cap X$.
\end{defn}

Note that this definition includes the case of a quadritangent. 
We are interested in the scheme which parametrises lines $l\subset\mP^3$ which are bitangents to $X$.
\begin{defn} We call
\begin{equation}\label{superficiebitangenti}
S:=\{ [l]\in\mathbb G\mid l\, {\rm{is}}\,  {\rm{bitangent}}\, {\rm{to}}\, X\}.
\end{equation}
{\it{the variety of bitangents to $X$}}.
\end{defn}
\noindent The geometry of Fano threefolds is strongly intertwined with the one of $S$.

\subsection{Lines on the quartic double solid} Let $Q$ be the $2$-to-$1$ cover of $\mP^3$ branched over $X$. We consider the tautological divisor $T_{\mP/\mP^3}$ of $\mP:=\mP(\sO_{\mP^3}\oplus \sO_{\mP^3}(2))$ and the natural projection $\rho_{\mP}\colon \mP\to \mP^3$.
Then $\rho_{\mP\star}\sO_{\mP}(T_{\mP/\mP^3})=\sO_{\mP^3}\oplus \sO_{\mP^3}(2)$. Let $H_{\mP^3}$ the hyperplane section of $\mP^3$. We set: 
$\sO_{\mP}(n):=\rho_{\mP}^{\star}\sO_{\mP^3}(nH_{\mP^3})$. If $T_1\in H^0(\mP, \sO_{\mP}(T_{\mP/\mP^3}))$ and 
$T_{\infty}\in H^0(\mP, \sO_{\mP}(T_{\mP/\mP^3}\otimes_{\sO_\mP}\sO_{\mP}(-2))$ it is easy to show that $Q\in |2T_{\mP/\mP^3}|$ and 
$$
Q= T_1^2-F(x_0,x_1,x_2,x_3)T_{\infty}^2=0.
$$

\noindent By the standard theory of double covers applied to $\rho: Q\to {\mP^3}$  it follows:
\begin{lem}\label{invariantideltrifoglio} For the Hodge numbers of $Q$ it holds: $h^{i,j}(Q)=0$ if $i\neq j$, except $h^{1,2}(Q)=h^{2,1}(Q)=10$, $h^{i,i}(Q)=1$, $0\leq i\leq 3$. Moreover ${\rm{Pic}}(Q)= \rho^{\star}(H_{\mP^3})\cdot  \mathbb Z$.
\end{lem}
\begin{proof} See \cite[p.8]{W}.
\end{proof}
The anti-canonical divisor $-K_{Q}\sim \rho^{\star}(2H_{\mP^3})$ is ample and $Q$ is a Fano variety of index $2$. Hence there is the following natural notion of {\it{line}} of $Q$.

\begin{defn} A line of $Q$ is a connected subscheme $r\subset Q$ of pure dimension $1$ such that $r\cdot \rho^{\star}(H_{\mP^3})=1$.
\end{defn}

In \cite[Page 374]{T} there is a description of the lines of $Q$. Here we only recall that if $[l]\in S$ is a general point then $\rho^{\star}l=l^{1}\cup l^{2}$ where $l^{1},l^{2}$ are irreducible rational curves which mutually intersect into two points.Thanks to the polarisation on $Q$ given by  $\rho^{\star}(H_{\mP^3})$ we can construct the Hilbert scheme $S_X$ of lines of $Q$ and the reader can easily see a natural forgetful map: $$f\colon S_X\to S.$$

\subsection{Smoothness of the surface of bitangent lines}

The next Proposition is well-known, possibly since very long time. 
We include a proof of it because in both references quoted above, \cite{T} and \cite{W}, the same result is shown but in a different way. Indeed, the proof given in \cite{T} relies on an important result from \cite{Is} concerning the smoothness of the Hilbert schemes of Fano threefolds; while the one given in \cite{W} is obtained by proving first the smoothness of $S_X$ and then that $f\colon S_X\to S$ is an \'etale double cover; see \cite[Lemma 1.1 page 15]{W} and c.f. Proposition \ref{liscia} below. Our proof is very direct and elementary.
\begin{prop}\label{lisciezza} 
The scheme $S\subset\mathbb G$ is a smooth surface if $X$ contains no line.
\end{prop}
\begin{proof} 
We fix a line $l\subset \mP^3$ which is bitangent to $X$. W.l.o.g. we can assume that $l:=(x_2=x_3=0)$ and that the two points in $X\cap l$ are $P=(1:0:0:0)$, $P_{\lambda}=(1:\lambda:0:0)$ where we do not assume $\lambda\neq 0$ (the case $\lambda=0$ corresponds to a quadritangent line). Then 
$$
F(x_0:x_1:x_2:x_3)=x_{1}^2(x_1-\lambda x_0)^2 +x_2G(x_0:x_1:x_2:x_3)+x_3H(x_0:x_1:x_2:x_3)
$$
where $G,H\in \mathbb C[x_0,x_1,x_2,x_3]$ are homogeneous forms of degree $3$. We recall that by our generality asumption $l\not\subset X$.

We consider an open neighbourhood $U'\subset \mathbb G$ of $[l]$ and let $(u_0,u_1,u_2,u_3)$ be a regular parameterisation of $U'$ inside $\mathbb G$; this means that for points $[r]$ close to $[l]$ inside $U'$ the corresponding line $r$ in $\mP^3$ is given as follows:
$$
r=\{ (x_0:x_1: x_0u_0+x_1u_1:x_0u_2+x_1u_3) \mid (x_0:x_1)\in\mP^1\}\subset\mP^3.
$$
We look for conditions on the tangent vector $v:=(u_0 ,u_1,u_2, u_3)\in T_{[l]}\mathbb G$ to be inside the Zariski tangent space $(m_{S,[l]}/m_{S,[l]}^2)^{\vee}$ of $S$ at $[l]$. This means that if in $\mathbb C[x_0, x_1, u_0,u_1,u_2,u_3,\epsilon]$, where $\epsilon^2=0$, we write $$f(x_0:x_1; u_0,u_1,u_2,u_3,\epsilon)=F(x_0:x_1: \epsilon (x_0u_0+x_1u_1):\epsilon(x_0u_2+x_1u_3))$$ then it must exists a polynomial $q\in \mathbb C[x_0, x_1, u_0,u_1,u_2,u_3,\epsilon]$  such that $q$ has degree at most $2$ in the variables $x_0, x_1$ and $f=q^2$. Since $f(x_0 :x_1; u_0, u_1,u_2, u_3, \epsilon)=x_{1}^{2}(x_{1}-\lambda x_{0})^2 + \epsilon ( (x_0u_0+x_1 u_1)g(x_0:x_1) +( x_0u_2+x_1u_3) h(x_0:x_1))$ where $g(x_0:x_1):= G(x_0:x_1:0:0)$ and $h(x_0:x_1):= H(x_0:x_1:0:0)$ this is possible iff $x_{1} (x_{1}-\lambda x_{0})$ is a factor of  $(x_0u_0+x_1u_1)g(x_0:x_1)+(x_0u_2+x_1u_3)h(x_0:x_1)$. We distinguish now two cases: $\lambda=0$ and $\lambda\neq 0$. If  $\lambda\neq 0$ then we obtain that $v\in (m_{S,[l]}/m_{S,[l]}^2)^{\vee}$ iff

\[
\begin{cases}
u_0g(1:0)+u_2h(1:0)=0 \\
(u_0+\lambda u_1)g(1:\lambda)+(u_2+\lambda u_3)h(1:\lambda)=0
\end{cases}
\]
The above linear system has rank $\leq 1$ iff $P$ or $P_{\lambda}$ is a singular point of $X$, which never happens by assumption. If $\lambda=0$ the condition is equivalent to

\[
\begin{cases}
u_0g(1:0)+u_2h(1:0)=0 \\
u_0\frac{\partial}{\partial x_1}g(1:0)+u_1g(1:0)+u_2\frac{\partial}{\partial x_1}h(1:0)+u_3h(1:0)=0
\end{cases}
\]
and again the rank is less or equal to $1$ iff $P$ is a singular point.
\end{proof}

We recall the following:

\begin{prop}\label{liscia} The map $f\colon S_X\to S$ is a $2$-to-$1$ \'etale cover. In particular $S_X$ is a smooth surface.
\end{prop}
\begin{proof} See \cite[Lemma 1.1 and Corollary 1.3, page 18]{W} or \cite[Proposition 2.4]{T}, and in \cite[Proposition 3.1]{T}. See also \cite[Remark 2.2.9]{KPS}.   
\end{proof}

\subsection{Invariants for the surface of bitangents}
We need to recall the invariants of $S$. To this aim we consider the restriction to $S$ of the universal exact sequence (\ref{univesattasequenza}) of $\mathbb G$ via the natural inclusion $j_S\colon S\hookrightarrow \mathbb G$:
\begin{equation}\label{universalonS}
0\to \sS^{\vee}_S\to V\otimes\sO_S\to \sQ_S\to 0.
\end{equation}
We denote by $\pi_S\colon \mP(\sQ_S)\to S$ the pull-back of the natural projection ${\pi_{\mathbb G}}\colon\mP(\sQ)\to \mathbb G$.

\noindent
\subsubsection{Basic diagrams} 
We consider the following natural diagram of morphisms:
\begin{equation*}\label{diagrammabase}
\xymatrix { 
&&&\\
&&&
S_X \ar[d]^-{f}  &\\
&Y\ar[r]^-{j_Y}\ar[d]^-{i_Y}&\mP(\sQ_S)\ar[r]^-{    {\pi_S}}\ar[d]^-{j_{\sQ_S}}&S \ar[d]^-{j_{S}}&\\
&\mP(\Omega^1_X(1))\ar[r]^-{j_{   \Omega^1_X   }}\ar[d]^-{\rho_X}&\ar[d]^-{\rho_\mP^3}\mP(\Omega^1_{\mP^3}(1))=\mP(\sQ)\ar[r]^-{\pi_{\mathbb G}}\ar[d]&\mathbb G
&\\
&X\ar[r]^-{j_X}&\mP^3&
}
\end{equation*}
where we have used the well-known isomorphism $\mP(\Omega^1_{\mP^3}(1))\cong \mP(\sQ)$ and where we have
defined
$$
Y:=\{(p,[l])\in \mP(\Omega^1_X(1))\mid p\in X\, {\rm{and}}\, l\, {\rm{ is\, bitangent\, to\, X \,and}} \, p\in l\cap X \}.
$$
In the above diagram (\ref{diagrammabase})  the inclusion
$j_{\Omega^1_X}\colon \mP(\Omega^1_X(1))\to \mP(\Omega^1_{\mP^3}(1))$ admits the factorization:

\begin{equation}\label{diagrammasullasuperficie}
\xymatrix{
&\mP(\Omega^1_X(1))\ar[r]^-{J_X}  \ar[d]^{\rho_X}&\mP(\Omega^1_{\mP^3|X}(1))\ar[r]^-{J_{\Omega^1_{\mP^3|X}}}  \ar[d]^{p_X}&\subset\mP(\Omega^1_{\mP^3}(1)) \ar[d]^{\rho_{\mP^3}}  &\\
&X\ar[r]^-{{\rm{id}}_X}&X\ar[r]^-{j_X}&\mP^3&
}
\end{equation}
\noindent 
where the inclusion $J_X\colon \mP(\Omega^1_X(1))\to \mP(\Omega^1_{\mP^3|X}(1))$ is given by the projectivisation of the standard conormal sequence:

\begin{equation}\label{conormal}
0\to \sO_X(-4)\to \Omega^1_{\mP^3|X}\to \Omega^1_X\to 0.
\end{equation}

Clearly the identification $\mP(\Omega^1_{\mP^3}(1))\cong \mP(\sQ)$ is given by the standard incidence variety. We have:

\begin{lem}\label{projcotangent} For the $3$-fold $\mP(\Omega^1_X(1))$ it holds:
$$\mP(\Omega^1_X(1))\cong \{(p,[l])\in X\times \mathbb G\mid l\,\, {\rm {is\,\, a\,\, tangent\,\, line\,\, to}}\,\, X\,\, {\rm{at}}\,\, P \}$$
\end{lem}
\begin{proof} Trivial since $X$ is smooth.
\end{proof}

\subsubsection{Useful divisor classes}

We denote by $N$, $R$, $T$ the tautological divisors on respectively $\mP(\Omega^1_{\mP^3}(1))$,
 $\mathbb P(\sQ)$, $\mP(\Omega^1_X)$; that is:
 \medskip 
\begin{equation*}\label{notazionesemplice}
\rho_{\mP^3\star}\sO_{\mP(\Omega^1_{\mP^3}(1))}(N)=\Omega^1_{\mP^3}(1),\, \pi_{\mathbb G\star} \sO_{\mP(\sQ)}(R)=\sQ,
\,\,\, \rho_{X\star}(\sO_{\mP(\Omega^1_X)}(T))=\Omega^1_X.
\end{equation*}
Since no confusion can arise, we denote by $R$ also the restriction to $\mP(\sQ_S)$ of $R$, hence we can write too: 
\begin{equation}\label{eillepullabacche}
\pi_{S\star}\sO_{\mP(\sQ_S)}(R)=\sQ_S.
\end{equation}

\begin{rem} By diagram (\ref{diagrammabase}) we can switch from divisor classes which are easily seen by the geometry over $\mP^3$ to ones which are easily seen by the geometry over $\mathbb G$ and viceversa.
\end{rem}

Letting $H_{\mP^3}$ the hyperplane section of $\mP^3$, since $R=\rho_{\mP^3}^{\star}H_{\mP^3}$ it is well-defined the restriction of $R$ to $\mP(\sQ_S)$, $\mP(\Omega^1_X(1))$,  $\mP(\Omega ^1_{\mP^3|X})$. We still denote by $R$ all these different restrictions since no confusion can arise. 
We can also restrict the pull-back $\pi^{\star}_{\mathbb G}(H_{\mathbb G})$ on $\mP(\sQ)\cong\mP(\Omega^1_{\mP^3}(1))$ to 
$\mP(\sQ_S)$, $\mP(\Omega^1_X(1))$,  $\mP(\Omega ^1_{\mP^3|X})$. We set $h:=H_{\mP^3\mid X}\in {\rm{Div}}(X)$, $H_X:= \pi^{\star}_{\mathbb G}(H_{\mathbb G})_{|\mP(\Omega^1_X(1))}\in {\rm{Div}}(\mP(\Omega^1_X(1))$.

\begin{lem}\label{relativeonX}  It holds on ${\rm{Pic}}(\mP(\Omega^1_X(1))):$
\begin{enumerate} 
\item $R\sim \rho_{X}^{\star}h$;
\item $N_{|\mP(\Omega^1_X(1))}\sim T +\rho_{X}^{\star}h$;
\item $H_X\sim T+2\rho_{X}^{\star}h$.
\end{enumerate}
\end{lem}
\begin{proof} See \cite[Formulae (1.2) p. 374]{T}.
\end{proof}

Inside $S$ there is the subscheme $B_{\rm{hf}}\hookrightarrow S$ which parametrises the hyperflex lines. In \cite[p. 18]{W} it is shown that $B_{\rm{hf}}$ is a smooth curve if $X$ is general. Actually it holds much more:
\begin{prop}\label{ilrivestimento doppio ramificato}
The non trivial $2$-torsion element $\sigma\in {\rm{Pic}}(S)$ associated to the covering $f\colon S_X\to S$ is such that $Y$ can be realised as a subscheme of $\mP(\sO_S\oplus\sO_S(\sigma+H_{\mathbb G|S}))$. The restriction of the natural projection $\mP(\sO_S\oplus\sO_S(\sigma+H_{|S}))\to S$ induces a $2$-to-$1$ cover $\pi\colon Y\to S$ branched over $B_{\rm{hf}}\in |2H_{\mathbb G|S}|$. In particular $Y$ is a smooth surface. Moreover as a divisor in
$\mP(\sQ_S)$ we have that $Y\in |2R+\pi_S^{\star}(\sigma)|$.
\end{prop}
\begin{proof} See \cite[Proposition 3.11]{W}.
\end{proof}

By the diagram in (\ref{diagrammabase}) we obviously obtain the following one:

\begin{equation}
\label{oneuniversal&projection}
\xymatrix{ & Y \ar[dl]_{{ \rho}}
\ar[dr]^{\pi}\\
 X\subset\mP^{3} &  & S\subset\mathbb G }
\end{equation}
We recall that since $X$ is general ${\rm{Pic}}(X)=[h]\mathbb Z$, hence 
\begin{equation}\label{picdi}
{\rm{Pic}}(\mP(\Omega^1_X(1)))=[T]\mathbb Z\oplus [R]\mathbb Z.
\end{equation}

\begin{prop}\label{sopraX}
As a divisor in  ${\rm{Pic}}(\mP(\Omega^1_X(1)))$ it holds that $$Y\in |6T+8R|.$$
\end{prop}
\begin{proof} See \cite[Proposition 2.3]{T}.
\end{proof}
\begin{rem}\label{swiTph} The surface $Y$ dominates both $X$ and $S$. Hence we can switch from divisor classes on $X$ to ones on $S$ and viceversa via divisors on $Y$.
\end{rem}

\subsubsection{The class of $S$ in the Chow ring of $\mathbb G$}
It is quite natural to introduce the following classes inside the Chow ring $\oplus_{i=1}^{4}{\rm{CH}}^{i}(\mathbb G)$ associated to the fundamental ladder, point, line, plane, $p\in l\subset h\subset \mP^3$:
${\rm{CH}}^1(\mathbb G)\ni\sigma_{l}:=\{[m]\in\mathbb G| m\cap l\neq \emptyset \}$, 
${\rm{CH}}^2(\mathbb G)\ni\sigma_{p}:=\{[l]\in\mathbb G| p\in l\}$, ${\rm{CH}}^2(\mathbb G)\ni \sigma_{h}:=\{[l]\in \mathbb G| l\subset h\}$. It is well known that 
$$\sigma_l^2=\sigma_h+\sigma_p$$
and that the divisorial class of $\sigma_l$ is the class $H_{\mathbb G}$.
\begin{lem}\label{chowclass} The following identity holds in the Chow ring of $\mathbb G$:
$$
{\rm{CH}}^2(\mathbb G)\ni[S]=28\sigma_h+12\sigma_p.
$$
In particular  $\rm{deg}(S)=H_{\mathbb G}^2\cdot [S]=40.$
\end{lem} 
\begin{proof}
See \cite[Lemma 3.30]{W}.
\end{proof}

\subsubsection{Numerical invariants}

\begin{thm}\label{formulae} For the surfaces $S_X$, $S$, $Y$ we have the formulae:
\begin{enumerate}
\item $K_{S_X}=f^*(3H_{\mathbb G|S})$, $q(S_X)=10$, $p_g(S_X)=101$, $h^1(S_X,\Omega^1_{S_X}) =220$, $c_2(S_{X})=384$,
\item  $K_{S}=3H_{\mathbb G|S}+\sigma$, $q(S)=0$, $p_g(S)=45$, $h^1(S,\Omega^1_{S}) =100$, $c_2(S)=192$
\item $K_Y=\pi^{\star}(4H_{\mathbb G|S})$, $q(Y)=0$, $p_g(Y)=171$
\end{enumerate}
\end{thm}
\begin{proof} See \cite[Cohomological study pp. 41-45]{W}.
\end{proof}

\subsection{The Abel-Jacobi morphism}

For completeness, below, we point out a property of  the Albanese morphism $a_X\colon S_X\to {\rm{Alb}}(S_X)$. Indeed, in the sequel we will only use that $S_X$ does not contain any curve birational to $\mP^1$ and this easily follows by \cite[Proposition 2.13 p. 27]{W}.
 
\begin{thm}\label{collino} If $X$ is general then the differential of the Albanese morphism $a_X\colon S_X\to {\rm{Alb}}(S_X)$ at any point $[l]\in S_X$ is injective.
 In particular $S_X$ does not contain any rational curve.
\end{thm}
\begin{proof} Let $J(Q)$ be the intermediate Jacobian of $Q$. By \cite[Proposition 2.13 p. 27]{W} we know that the differential of the Abel-Jacobi map $S_X\to J(Q)$ is injective. By \cite[Theorem 4.1]{W} the Abel-Jacobi morphism ${\rm{Alb}}(S_X)\to J(Q)$ is an isomorphism. By the universal property of the Albanese morphism it holds that 
the differential of the Albanese morphism $a_X\colon S_X\to {\rm{Alb}}(S_X)$ is injective. Now let $C\subset S_X$ be a curve such that its normalisation is $\nu\colon \mP^1\to C$. 
Then the image of $C$ by $a_X\colon S_X\to {\rm{Alb}}(S_X)$ is a point. Finally let $[l]$ be a general point of $C$. Then $[l]$ is a smooth point of $C$ but 
${\rm{dim}}_{\mathbb C}{\rm{Ker}}(da_{X,[l]}) \geq 1$ since the tangent direction to $C$ at $[l]$ is sent to $0$ . A contradiction.
\end{proof}

\subsection{Rational curves on the surface of bitangent lines}

\begin{thm}\label{norational}  If $X$ is a smooth quartic with no line contained in it then there are no rational curves on $S$.
\end{thm}
\begin{proof} Let $f\colon S_X \to S$ be the \'etale $2$-to-$1$ covering of $S$ given in Proposition \ref{liscia}. Let $C\hookrightarrow S$ be a curve such that its normalisation is $\nu\colon \mP^1\to C$. It holds that the normalisation of $f^{*}C$ is a union of rational curves. On the other hand since $X$ is general, by Theorem \ref{collino} it holds that these curves cannot exist.
\end{proof}

\section{A counterexample to infinitesimal Torelli}
The surface $S$ comes equipped with a special infinitesimal deformation class which gives a counterexample to infinitesimal Torelli.

\subsection{Torelli infinitesimal deformation}
By  \cite[Theorem 4]{T} or \cite[Proposition 3.15]{W} the tangent bundle sequence gives
\begin{equation}\label{simsigma}
0\to\sO_S\to{\rm{Sym}}^2 (\sQ_S) \otimes \sO_S (\sigma)\to \Omega^1_S\to 0
\end{equation}
where $\sigma\in {\rm{Pic}}^2(S) \setminus\{ 0\}$ is described in Proposition \ref{ilrivestimento doppio ramificato}. We consider the extension class $\xi$ of the sequence (\ref{simsigma}): $\xi\in H^1(S, T_S)$. This class gives two homomorphisms:
$\partial^{(1)}_{\xi}\colon H^0(S, \Omega^1_S)\to H^1(S,\sO_S)$, which is obviously the trivial one since $q(S)=0$, and 
$$\partial^{(2)}_{\xi}\colon H^0(S, \omega_S)\to H^1(S, \Omega^1_S).$$
\noindent They are respectively the first coboundary homomorphism of the sequence (\ref{simsigma}) and the first coboundary homomorphism of the following sequence:
 
\begin{equation}\label{wedgesimsigma}
0\to \Omega^1_S\to \bigwedge^2\bigl ({\rm{Sym}}^2(\sQ_S)\otimes\sO_S(\sigma)\bigr )\to\omega_S\to 0
\end{equation}

We are going to show that $\partial^{(2)}_{\xi}$ is the trivial homomorphism too. First we show:

\begin{thm}\label{infty} The class $\xi\in H^1(S, T_S)$ of the tangent bundle sequence is non zero.
\end{thm}
\begin{proof} By contradiction assume that $\xi= 0$. Then even the dual sequence
\begin{equation}\label{simsigmadual}
0\to\sT_S\to{\rm{Sym}}^2 (\sQ_S^{\vee}) \otimes \sO_S (\sigma)\to \sO_S\to 0
\end{equation}
splits. Hence $h^0(S, {\rm{Sym}}^2 (\sQ_S^{\vee}) \otimes \sO_S (\sigma))=1$. By Lemma \ref{ARRONDOondo} $(3)$ we have the natural isomorphism ${\rm{Sym}}^2 (\sQ^{\vee})={\rm{Sym}}^2(\sQ) \otimes \sO_\mathbb G (-2)$ over $\mathbb G$. It restricts over $S$ to the isomorphism ${\rm{Sym}}^2 (\sQ_S^\vee)=({\rm{Sym}}^2(\sQ_S))(-2)$. Up to now we have shown that if the sequence (\ref{simsigmadual}) splits then it follows that:
$$
h^0(S, ({\rm{Sym}}^2(\sQ_S))(-2) \otimes \sO_S (\sigma))=1.
$$
 On the other hand if we tensor the sequence (\ref{simsigma}) by $\sO_S(-2)$ we obtain
$$
0\to\sO_S(-2)\to({\rm{Sym}}^2(\sQ_S))(-2) \otimes \sO_S (\sigma)\to \Omega^1_S(-2)\to 0
$$
which clearly implies $h^0(S, ({\rm{Sym}}^2(\sQ_S))(-2) \otimes \sO_S (\sigma))=0$ since $q(S)=0$. We have shown $\xi\neq 0$.
\end{proof}

\begin{thm}\label{regularcase} $\partial^{(2)}_{\xi}=0$.
\end{thm}
\begin{proof} 
By Theorem \ref{formulae} we have that $p_g(S)=45$.
\noindent Since $$
\bigwedge^2\bigl ({\rm{Sym}}^2(\sQ_S)\otimes\sO_S(\sigma)\bigr )=\bigwedge^2{\rm{Sym}}^2(\sQ_S)$$ then we have only to show that $h^0(S,\bigwedge^2{\rm{Sym}}^2(\sQ_S))\geq 45$. On the other hand by \cite[Page 46]{W} there is an injection: $$\bigwedge^2H^0(\mP^3,\sO_{\mP^3}(2))\hookrightarrow H^0(S,\bigwedge^2{\rm{Sym}}^2(\sQ_S)).$$ Since ${\rm{dim}}\bigwedge^2H^0(\mP^3,\sO_{\mP^3}(2))=45$ the claim follows.
\end{proof}
We have shown the following
\begin{thm}\label{primooo} The surface $S\subset\mathbb G$ is smooth, regular, of general type, it contains no rational curve and it is a counterexample to  infinitesimal Torelli. 
\end{thm}
\begin{proof} It follows straightly by Proposition \ref{lisciezza}, Theorem \ref{formulae}, Theorem \ref{norational},Theorem \ref{infty} and Theorem \ref{regularcase}.
\end{proof}

\section{The canonical map of the surfaces of bitangents}

We study the canonical map $\phi_{|K_S|}\colon S\dashrightarrow \mP(H^0(S,\sO_S(K_S))^{\vee})$.
\subsection{Very ampleness of the canonical system}
To show that $$\phi_{|K_S|}\colon S\to\mP(H^0(S,\sO_S(K_S))^{\vee})$$ is an embedding let us consider the following rank $3$ vector bundle $({\rm{Sym}}^2(\sQ))^{\vee}$ and the associated projective bundle $\pi_3\colon \mP({\rm{Sym}}^2(\sQ))\to\mathbb G$. Geometrically a point $\eta\in \mP({\rm{Sym}}^2(\sQ))$ is the datum of a point $[l]\in\mathbb G$ and of a set of two points $P,P'\in l$ where $P$ is not necessarily distinct from $P'$. We consider the map $\Phi\colon \mP({\rm{Sym}}^2(\sQ))\to \mP(\bigwedge^2({\rm{Sym}}^2V)^{})$ which to a point  $ \mP({\rm{Sym}}^2(\sQ))\ni \eta=([l],P,P')$ associates annihilator of the pencil of quadrics inside $\mP(V)$ containing the two points $P,P'$ or, in the case where $P=P'$ it associates annihilator of the space of quadrics $A$ passing through $P$ and such that $T_PA$ contains the line $l$.

\begin{prop}\label{applicazione} The map  $\Phi\colon \mP( {\rm{Sym}}^2(\sQ) )\to \mP(\bigwedge^2 {\rm{Sym}}^2V^{} )$ is an embedding.
\end{prop}

We note that six-dimensional variety  $\mP( {\rm{Sym}}^2(\sQ))$ can be viewed as the Hilbert scheme $\mathcal{H}ilb_2(\mP^3)$ parametrizing unordred pairs of distinct points and points with a tangent direction. Indeed $\mP( {\rm{Sym}}^2(\sQ))$ is the bundle over $\mG$ formed by the pairs $(l,p_1+p_2)$, where $l\in\mG$ is a line in $\mP^3$ and $p_1+p_2$ is an effective degree two divisor on $l$. The group $G=\mP \mathrm{GL}(V)\simeq \mP \mathrm{GL}(4)$ acts canonically on $\mP^3$; its action induces corresponding actions on $\mathcal{H}ilb_2(\mP^3)$ and on the target space $\mP(\bigwedge^2 {\rm{Sym}}^2V)$ for the morphism $\Phi$. Clearly, the map $\Phi$  is equivariant, which means that for every $x\in \mathcal{H}ilb_2(\mP^3)$ and every $g\in G$, $\Phi(g(x))=g(\Phi(x))$.

The above Proposition \ref{applicazione} then follows from the more general proposition below, which might be of some independent interest:

\begin{prop}\label{equivariante} Let $\Phi: \mathcal{H}ilb_2(\mP^3)\to \mP^N$ be a $G$-equivariant morphism (with respect to some action of $G=\mP \mathrm{GL}(4)$ on $\mP^N$). If $\Phi$ is injective, then it is an embedding.
\end{prop}

(We owe to an anonymous referee the idea for the following proof).

\begin{proof}
 The six-dimensional variety $\mathcal{X}:=\mP( {\rm{Sym}}^2(\sQ) )\simeq \mathcal{H}ilb_2(\mP^3)$ decomposes as the union of two orbits under $G$: namely the closed orbit $\mathcal{Y}$ formed by the pairs of the form $(l,2p)$ with $p\in l\in \mG$, and its complement $\mathcal{X}\setminus\mathcal{Y}$. It is clear that the differential of $\Phi$ has maximal rank at every point of the open set $\mathcal{X}\setminus\mathcal{Y}$.
 
 Hence we restrict our attention to an arbitrary point $(l_0,2p_0)\in\mathcal{Y}$. Our aim is proving that the differential of $\Phi$ at this point has maximal rank (i.e. rank six). 
 
We define the following closed subvarieties of $\mathcal{X}$:
\begin{equation*}
\begin{matrix}
  \mathcal{X}_{l_0}&=&\{(l_0,p_1+p_2)\, |\,  p_1,p_2\in l_0\}\\
  \mathcal{X}_{p_0} &=& \{(l,p_0+p)\, |\, l\ni p_0, p\in l\}\\
  \mathcal{Z}_{l_0} &=& \{(l,p_1+p_2)\in\mathcal{X}\, |\, l\cap l_0\neq\emptyset\}\\
 \mathcal{X}_{l_0, p_0} &=& \mathcal{X}_{l_0}\cap \mathcal{X}_{p_0}
\end{matrix}
\end{equation*}
We also define the subgroups $G_{l_0}\subset G$ to be the stabilizer of the line $l_0$, $G_{p_0}$  the stabilizer of the point $p_0$ and finally set $G_{l_0,p_0}=G_{l_0}\cap G_{p_0}$.

The key idea of the proof is that the kernel of the differential of $\Phi$ at $(l_0,2p_0)$ must be invariant under the natural action of the group $G_{l_0,p_0}$ on the tangent space $T_{(l_0,2p_0)}(\mathcal{X})=:W$. 

We note that all the above defined subvarieties (including $\mathcal{Y}$) are invariant under $G_{l_0,p_0}$, so their tangent spaces are invariant too. Also, every other $G_{l_0,p_0}$-invariant subspace of $W$ is contained in the tangent space of one of the above varieties. Also, the $G_{l_0,p_0}$-action on $W$ admits an open orbit $W^0$. If this orbit contains a (necessarily non-zero) vector of the kernel, then the differential at $(l_0,2p_0)$ would vanish identically. However, this is not the case, as we now show: consider the subvariety $\mathcal{Y}\subset\mathcal{X}$: as we noticed, it is a homogeneous space for the group $G$ and   the restriction of $\Phi$ to this variety is a $G$-invariant injective morphism, so it is an embedding. This shows in particular that the differential at $(l_0,2p_0)$ of $\Phi$ cannot vanish at any non-zero vector of $T_{(l_0,2p_0)}(\mathcal{Y})\subset W$. 

It remains to exclude that the differential vanishes on some  non-zero vector tangent to $\mathcal{X}_{l_0}, \mathcal{X}_{p_0}$ or $\mathcal{Z}_{l_0}$. 

Consider first the case of the two-dimensional variety $\mathcal{X}_{l_0}$: it contains two invariant curves isomorphic to $\mP^1$, namely $\mathcal{X}_{l_0}\cap \mathcal{Y}$ and $\mathcal{X}_{l_0, p_0}$. The first one is a homogeneous space under $G_{l_0}$, so that the injective morphism $\Phi$ restricted to this curve is an embedding. The second one is not really a homogeneous space; however the group $G_{l_0,p_0}$ acts transitively on $\mathcal{X}_{l_0,p_0}\setminus \{(l_0,2p_0)\}\simeq \mA^1$ as the group of affine transformations. To prove that the differential cannot vanish at the (special) point $(l_0,2p_0)$ we use the following general result:

\begin{lem}\label{P^1-equivariante}
Let $\Ga\times\mP^N\to\mP^N$ be a non-trivial action of the additive group $\Ga$ on a projective space.
Let $\Phi:\mP^1\to\mP^N$ be an injective morphism which is equivariant under the additive group $\Ga$, acting in the usual way on $\mP^1$. Then the curve $\Phi(\mP^1)$ is smooth and the map $\Phi$ is an embedding.
\end{lem}
\begin{proof}
The possible actions of $\Ga$ on $\mP^N$ are given by unipotent matrices: namely from a unipotent matrix $T\in\SL_{N+1}(\mC)$ one constructs an action by setting $\Ga\times \mP^N\ni (t,[v])\mapsto [T^t\cdot v] \in\mP^N$, where the exponentiation is well defined since $T$ is unipotent. It follows from the classification of unipotent matrices that the closure of the orbit of any non-fixed point under such an action is a smooth curve. Actually, under a suitable projection onto a subspace,  such a curve is sent to a rational normal curve on that subspace.
\end{proof}

Since our map is equivariant under the full group of affine transformations of the line it is also equivariant under its subgroup $\Ga$ of translations; hence the  above lemma applies and prove the non-vanishing of the differential of the restriction of $\Phi$ to $\mathcal{X}_{l_0,p_0}$. 

We have then proved that the differential of $\Phi$ at $(l_0,p_0)$ cannot vanish in the tangent  directions to the two invariant subvarieties. Since the other tangent directions form a homogeneous space for the group $G_{l_0}$, the differential cannot vanish on any of these directions, hence it has no   zero on any non-zero vector tangent to $\mathcal{X}_{l_0}$.

Consider now the case of the three-dimensional variety $\mathcal{X}_{p_0}$.  It is invariant under the group $G_{l_0}$ and contains the two invariant subvarieties $\mathcal{X}_{l_0,p_0}$ and $\mathcal{X}_{p_0}\cap \mathcal{Y}$. The first one, as mentioned, is isomorphic to $\mP^1$, the second one to $\mP^2$. We have already proved that the differential of $\Phi$ cannot vanish on the direction tangent to $\mathcal{X}_{l_0,p_0}$; concerning the second invariant subvariety, it is a homogeneous space under $G_{p_0}$, hence again the differential cannot vanish on any of its tangent directions. We then conclude as before, since all the other tangent directions form a unique orbit for the action of $G_{p_0}$.

Consider now the fifth-dimensional invariant subvariety $\mathcal{Z}_{l_0}$. Its maximal invariant subvarieties are the previously considered subvarieties $\mathcal{X}_{l_0}$ and $\mathcal{X}_{p_0}$. The tangent directions to $\mathcal{Z}_{l_0}$ not tangent to any of these two subvarieties form again a unique orbit for the group $G_{l_0}$. From what we have just proved, it follows that the differential of $\Phi$ at $(l_0,p_0)$ has no non-trivial zero on the tangent space  $T_{(l_0,p_0)}(\mathcal{Z}_{l_0})$. 

Finally, consider the full tangent space $W$. Its maximal invariant subspaces are $T_{(l_0,p_0)}(\mathcal{Z}_{l_0})$ and $T_{(l_0,p_0)}(\mathcal{Y})$, where we have proved that the differential has no non-trivial zero. Once again, the other tangent directions form a unique orbit for the group $G_{l_0,p_0}$, so that the differential cannot have any non-trivial zero at all.

This concludes the proof of Proposition \ref{equivariante} and of Proposition \ref{applicazione}.
\end{proof}

\begin{thm}\label{veryampleness}  If $X$ is a general quartic then the canonical sheaf of $S$ is very ample.
\end{thm}
\begin{proof} By \cite[ (A.5) Proposition p.53]{W} we know that the canonical map of $S$ is given by $S\to  \mP(\bigwedge^2{\rm{Sym}}^2V^{} ) \to \mathbb G(2, ({\rm{Sym}}^2V)^{} )$, where the last map is Pl\"ucker embedding and the fisrt one is defined as  follows 
\begin{equation*}
l \mapsto [{\rm{Ann}}(\{q\in  {\rm{Sym}}^2V^{\vee})\mid Z(q_{\mid l})=l_{|X}\}]
\end{equation*}
(here we used the symbol $Z(\cdot)$ to denote the zero set of a function, or a section of a bundle).
 Now the basic remark is that by the identity $({\rm{Sym}}^2(\sQ_S))^{\vee}=({\rm{Sym}}^2(\sQ_S))\otimes_{\sO_{\mathbb S}}\sO_{S}(-jH_{\mathbb G|S})$ we have that $\mP(({\rm{Sym}}^2(\sQ_S))^{\vee})\simeq \mP({\rm{Sym}}^2(\sQ_S))$. The surjective morphism ${\rm{Sym}}^2(\sQ_S))^{\vee})\otimes_{\sO_S}\sO_S(\sigma)\to \sO_S$ given by the tangent sequence induces a section $$f_{\sigma}\colon S\hookrightarrow \mP({\rm{Sym}}^2(\sQ_S))^{\vee}\otimes_{\sO_S}\sO_S(\sigma))=\mP({\rm{Sym}}^2(\sQ_S))$$
which composed with the embedding $\mP({\rm{Sym}}^2(\sQ_S))\hookrightarrow  \mP( {\rm{Sym}}^2(\sQ) )$ and the embedding $\Phi\colon \mP( {\rm{Sym}}^2(\sQ) )\to \mP(\bigwedge^2 {\rm{Sym}}^2V^{} )$ of Proposition \ref{applicazione} shows that the canonical map of $S$ is an embedding.
\end{proof}

\begin{thm}\label{teoremaA} The surface $S\subset\mathbb G$ which parametrises the bitangent lines to a general quartic $X\subset\mP^3$ is smooth, regular, of general type, it contains no rational curve, it does not satisfie iTp  and its canonical map is an embedding.
\end{thm}
\begin{proof} It follows by Theorem \ref{primooo} and Theorem \ref{veryampleness}.
\end{proof}

\section{The canonical image of the surfaces of bitangents}
We study the canonical image of $S$, but since the canonical map is an embedding we do not distinguish here between $S$ and its canonical image. In particular we show that
the standard multiplicative map
\begin{equation}\label{standmultmap}
\mu\colon H^0(S,\omega_S)\otimes_\mathbb C H^0(S,\omega_S)\to H^0(S,\omega^{\otimes 2}_S)
\end{equation}
is not surjective.

\subsection{Ample sheaves on $S$}

\begin{lem}\label{finitezza} The morphism $\rho_{\mP^3}\circ j_{\sQ_S}\colon\mP(\sQ_S)\to\mP^3$ is finite of degree $12$
\end{lem}
\begin{proof} By contradiction. Let $p\in\mP^3$ such that the $\rho_{\mP^3}\circ j_{\sQ_S}$-fiber is of positive dimension. This means that there are infinitely many  bitangent lines through $p$. Then the polar cubic ${\rm{Pol}}_{p}(X)$ has at least a component swept by lines through $p$. The restriction of ${\rm{Pol}}_{p}(X)$ to $X$ is non reduced. This implies that this restriction is a divisor of type $2D+A$. Since we are assuming that ${\rm{Pic}}(X)=[H_{ \mP^{3}_{\mid X} }] \mathbb Z$, $X$ does not contain curves of degree $\leq 3$. Then the only possibility is that $D$ is an hyperplane section. This  implies that $S$ contains the rational curve which parametrises the pencil of lines contained in a plane and passing through $p$. By Theorem \ref{norational} this is a contradiction. Finally by Lemma \ref{chowclass} it follows that:
$$
{\rm{deg}}(\rho_{\mP^3}\circ j_{\sQ_S})=[S]\cdot\sigma_p=(28\sigma_h+12\sigma_p)\cdot\sigma_p=12.
$$

\end{proof}
\begin{cor} For every $j\geq 1$ the sheaves ${\rm{Sym}}^j(\sQ_S)$ are ample.
\end{cor}
\begin{proof} Since we work in characteristic zero, by \cite[Corollary 5.3 page 77]{H} we only need to show that $\sQ_S$ is ample. The tautological sheaf $R$ of $\mP(\sQ_S)$ is the pull-back of the hyperplane section $H_{\mP^3}$ by the morphism $\rho_{\mP^3}\circ j_{\sQ_S}\colon\mP(\sQ_S)\to\mP^3$, see Lemma \ref{relativeonX} and the discussion above it. By Lemma \ref{finitezza} $\rho_{\mP^3}\circ j_{\sQ_S}$ is a finite morphism. Hence the claim follows.
\end{proof}
\begin{rem} We observe that over the Grassmannian $\mathbb G$ the bundle $\sQ$ is only generated but not ample.
\end{rem}

\begin{cor}\label{farezero} The sheaf ${\rm{Sym}}^2(\sQ_S)\otimes_{\sO_S}\sO_S(1)$ is Nakano positive. In particular
$$
H^1(S, {\rm{Sym}}^2(\sQ_S)\otimes_{\sO_S}\sO_S(1)\otimes_{\sO_S}\omega_S)=0.
$$
\end{cor}
\begin{proof} Since $\sO_S(1)=\bigwedge^2\sQ_S$ the claim follows straightly by \cite[Proposition 1.5]{LY}.
\end{proof}

\subsection{The canonical image is not $2$-normal}
We consider again the sequence (\ref{wedgesimsigma}), but since by Corollary \ref{wedge2sim2} 
\begin{equation}\label{twistachetipassa}
\bigwedge^2\bigl ({\rm{Sym}}^2(\sQ_S)\otimes\sO_S(\sigma)\bigr )=\bigwedge^2{\rm{Sym}}^2(\sQ_S)\simeq {\rm{Sym}}^2(\sQ_S)\otimes_{\sO_S}\sO_S(1).
\end{equation}
we can write it in the following form:

\begin{equation}\label{wedgesimsigmasigma}
0\to \Omega^1_S\to {\rm{Sym}}^2(\sQ_S)\otimes_{\sO_S}\sO_S(1)\to\omega_S\to 0
\end{equation}

We tensor  it by $\otimes_{\sO_S}\omega_S$ and by a suitable twist of the dual Euler sequence induced by 
the canonical embedding $\phi_{|K_S|} \colon S \to \mP^{44}$ we can built the following exact diagram:

\begin{equation*}
\xymatrix { 
&&&0\ar[d]\\
&&&
\phi_{|K_S|}^{\star}(\Omega^1_{\mP^{44}}(1))\ar[r]\ar[d]&0\\
0\ar[r]&\Omega^1_S\otimes H^0(S,\omega_S)\ar[r]\ar[d]^-{ {\rm{id_{\Omega^1_S}}}\otimes{\rm{ev}}}&{\rm{Sym}}^2(\sQ_S)\otimes_{\sO_S}\sO_S(1)\otimes H^0(S,\omega_S)\ar[r]^-{\rho\otimes {\rm{id}}}\ar[d]^-{{\rm{id_{}}}\otimes{\rm{ev}}}&\omega_S\otimes H^0(S,\omega_S)\ar[r]\ar[d]^-{{\rm{id_{\omega_S}}}\otimes{\rm{ev}}}&0\\
0\ar[r]&\Omega^1_S\otimes\omega_S\ar[r]\ar[d]&{\rm{Sym}}^2(\sQ_S)\otimes_{\sO_S}\omega_S(1)\ar[r]^-{\nu}\ar[d]&\omega_{S}^{\otimes 2}\ar[r]\ar[d]&0\\
&0&0&0
}
\end{equation*}
We observe that the vertical maps are surjective since $\omega_S$ is generated. 

Moreover it is well-known that for the map $\mu$ recalled in diagram (\ref{standmultmap}) it holds that:

$$\mu= H^0({\rm{id_{\omega_S}}}\otimes{\rm{ev}})\colon H^0(S,\omega_S\otimes H^0(S,\omega_S))\to H^0(S,\omega_{S}^{\otimes 2}).$$
\begin{thm}\label{nondue} The cokernel of $\mu$ has dimension at least equal to $h^1(S,T_S)$. In particular the canonical image of $S$ is not $2$-normal.
\end{thm}
\begin{proof} By Theorem \ref{regularcase}
$H^0(\rho)\colon H^0(S, {\rm{Sym}}^2(\sQ_S)\otimes_{\sO_S}\sO_S(1))\to H^0(S,\omega_S)$ is an isomorphism then
$$
H^0(\rho\otimes {\rm{id}})\colon H^0(S, {\rm{Sym}}^2(\sQ_S)\otimes_{\sO_S}\sO_S(1)\otimes H^0(S,\omega_S))\to
H^0(S,\omega_S)\otimes H^0(S,\omega_S)
$$
is an isomorphism. Then by the cohomology of the above diagram $${\rm{dim}}_{\mathbb C}{\rm{coker}}\mu\geq 
{\rm{dim}}_{\mathbb C}{\rm{coker}}H^0(\nu).$$

By Corollary \ref{farezero} $H^1(S, {\rm{Sym}}^2(\sQ_S)\otimes_{\sO_S}\omega_S(1))=0$. Then by Serre duality ${\rm{dim}}_{\mathbb C}{\rm{coker}}H^0(\nu)=h^1(S,T_S)$.  By Theorem \ref{infty} $h^1(S,T_S)\geq 1$. Then the claim follows.
\end{proof}

\section{Infinitesimal deformations of the surface of bitangents}
This section is entirely devoted to show that $$h^1(S, \Omega^1_S\otimes_{\sO_S}\omega_S)=h^1(S,T_S)=20.$$ 

We tensor  the tangent bundle sequence (\ref{simsigma}) by $\otimes_{\sO_S}\omega_S$ to obtain

\begin{equation}\label{tortatorta}
0\to\omega_S\to{\rm{Sym}}^2 (\sQ_S) \otimes \sO_S (\sigma)\otimes_{\sO_S}\omega_S\to \Omega^1_S\otimes_{\sO_S}\omega_S\to 0
\end{equation}

Since $q(S)=h^1(S,\omega_S)=0$, $45=h^0(S,\omega_S)$ to compute $h^0(S, \Omega^1_S\otimes_{\sO_S}\omega_S)$ is equivalent to compute $h^0(S,{\rm{Sym}}^2 (\sQ_S) \otimes \sO_S (\sigma)\otimes_{\sO_S}\omega_S)$. This computation relies on the strategy outlined in Remark \ref{swiTph}.

\begin{prop}\label{calcolo} It holds that $h^0(S, {\rm{Sym}}^2 (\sQ_S) \otimes \sO_S (\sigma)\otimes_{\sO_S}\omega_S)=325$.
\end{prop} 
\begin{proof} We consider $\mP(\sQ_S)$ and the diagram \ref{diagrammabase}. We recall that 
$${\rm{Sym}}^2 (\sQ_S) \otimes \sO_S (\sigma)\otimes_{\sO_S}\omega_S=
\pi_{S\star}\sO_{\mP(\sQ_S)}(2R+  \pi_S^{\star}(\sigma+\omega_S)).
$$
By Proposition \ref{ilrivestimento doppio ramificato} for the $2$-to-$1$ cover $\pi\colon Y\to S$ it holds that $Y\in |2R+\pi_S^{\star}(\sigma)|$. 
In particular $2R+  \pi_S^{\star}(\sigma+\omega_S)\sim Y+\pi_S^{\star}(\omega_S)$ Then we can write the following exact sequence on $\mP(\sQ_S)$:
\begin{equation}\label{soYYY}
0\to \sO_{\mP(\sQ_S)}( \pi_S^{\star}(\omega_S))\to\sO_{\mP(\sQ_S)}(Y+\pi_S^{\star}(\omega_S)) \to\sO_{Y}(Y+\pi_S^{\star}(\omega_S))\to0
\end{equation}
Since $R^1\pi_{S\star}\sO_{\mP(\sQ_S)}( \pi_S^{\star}(\omega_S))=0$, $h^0(S,\omega_S)=45$ and since $S$ is a regular surface it holds that 
$$
h^0(S, {\rm{Sym}}^2 (\sQ_S) \otimes \sO_S (\sigma)\otimes_{\sO_S}\omega_S)=45+h^0(\sO_{Y}(Y+\pi_S^{\star}(\omega_S))
$$
Since $\omega_S=\sO_S(3H_{\mathbb G}+\sigma)$ and $Y\in |2R+\pi_S^{\star}(\sigma)|$ we can also write that
$$
Y+\pi_S^{\star}(\omega_S)\sim \pi_S^{\star}(3H_{\mathbb G})+2R.
$$
To compute $h^0(\sO_{Y}(Y+\pi_S^{\star}(\omega_S))$ we switch to the geometry on $X$. Indeed by Proposition \ref{sopraX} we have that $Y\in |6T+8R|$ as a divisor of ${\rm{Pic}}(\mP(\Omega^1_X(1)))$. By Lemma \ref{relativeonX} we have $R\sim \rho_{X}^{\star}h$ and $\pi_S^{\star}(3H_{\mathbb G})_{|\mP(\Omega^1_X(1))}=3T+6\rho_{X}^{\star}h$. This means that the divisor $Y+\pi_S^{\star}(\omega_S)$ is induced on $Y$ by the divisor $3T+8\rho_{X}^{\star}h$. By the structure sequence of $Y$ inside $\mP(\Omega^1_X(1))$ we can write:

\begin{equation}\label{soXXXX}
0\to \sO_{\mP(\Omega^1_X(1))}(-3T)\to\sO_{\mP(\Omega^1_X(1))}(3T+8\rho_{X}^{\star}h)\to\sO_{Y}(3T+8\rho_{X}^{\star}h)\to 0
\end{equation}
Since the canonical divisor $K_{\mP(\Omega^1_X(1))}=-2T$, by Serre duality 
$$H^1(\mP(\Omega^1_X(1)), \sO_{\mP(\Omega^1_X(1))}(-3T))=H^2(\mP(\Omega^1_X(1)), \sO_{\mP(\Omega^1_X(1))}(T)).$$

Since $R^1\rho_{X\star}\sO_{\mP(\Omega^1_X(1))}(T)=0$ then $$H^2( \mP(\Omega^1_X(1)),\sO_{\mP(\Omega^1_X(1))}(T))=H^2(X,\Omega^1_X)=0.$$ In particular it holds that:
$$
h^0(\sO_{Y}(Y+\pi_S^{\star}(\omega_S))=\sO_{Y}(3T+8\rho_{X}^{\star}h)=H^0(\mP(\Omega^1_X(1)),\sO_{\mP(\Omega^1_X(1))}(3T+8\rho_{X}^{\star}h).
$$
By standard theory we know that $\rho_{X\star}\sO_{\mP(\Omega^1_X)}(3T+8\rho_X^{\star}(h))={\rm{Sym}}^{3}\Omega^1_X(8))$. Then by the symmetrisation of the co-normal sequence of $X$, see c.f. \cite[Lemma 2.9 page 19]{AK}, and suitable twisting of it we obtain:
\begin{equation}\label{symmsuX}
0\to {\rm{Sym}}^{2}\Omega^1_{\mP^3_{|X}}(4)\to{\rm{Sym}}^{3}\Omega^1_{\mP^3_{|X}}(8)\to{\rm{Sym}}^{3}\Omega^1_X(8)\to 0
\end{equation}
Now by the standard exact sequence which holds for every $n,m\in\mathbb Z$ :
\begin{equation}\label{vaiinrestrizione}
0\to{\rm{Sym}}^{m}\Omega^1_{\mP^3}(n-4)\to {\rm{Sym}}^{m}\Omega^1_{\mP^3}(n) \to{\rm{Sym}}^{m}\Omega^1_{\mP^3_{|X}}(n)\to0
\end{equation}
we can build the following diagram:

\begin{equation*}
\xymatrix { 
&0\ar[d]&0\ar[d]&\\
0\ar[r]& {\rm{Sym}}^{2}\Omega^1_{\mP^3_{}}\ar[d]&{\rm{Sym}}^{3}\Omega^1_{\mP^3_{}}(4)\ar[d]&&\\
0\ar[r]& {\rm{Sym}}^{2}\Omega^1_{\mP^3_{}}(4)\ar[d]&{\rm{Sym}}^{3}\Omega^1_{\mP^3_{}}(8)\ar[d]&&\\
0\ar[r]& {\rm{Sym}}^{2}\Omega^1_{\mP^3_{|X}}(4)\ar[r]\ar[d]&{\rm{Sym}}^{3}\Omega^1_{\mP^3_{|X}}(8)\ar[r]^-{\nu}\ar[d]&{\rm{Sym}}^{3}\Omega^1_X(8)\ar[r]&0\\
&0&0&
}
\end{equation*}
Since for every $n,m\in\mathbb Z$ it holds that:
\begin{equation}\label{simpenne}
0\to {\rm{Sym}}^{m}\Omega^1_{\mP^3}(n)\to {\rm{Sym}}^{m}V(-m+n)\to {\rm{Sym}}^{m-1}V(-m+n+1)\to 0
\end{equation}
where for every $n,m\in\mathbb Z$ we have set ${\rm{Sym}}^{m}V(n):= {\rm{Sym}}^{m}V\otimes_{\sO_{\mP^3}} \sO_{\mP^3}(nH_{\mP^3})$. By direct computation, or by c.f. \cite[Proposition 3.40 page 40]{W},
we obtain the following numerical results:

$$
h^0(\mP^3,{\rm{Sym}}^{3}\Omega^1_{\mP^3_{}}(8))=280,\,\, h^i(\mP^3,{\rm{Sym}}^{3}\Omega^1_{\mP^3_{}}(8))=0,\, \forall i>0
$$
$$
h^1(\mP^3,{\rm{Sym}}^{3}\Omega^1_{\mP^3_{}}(4))=20,\,\, h^i(\mP^3,{\rm{Sym}}^{3}\Omega^1_{\mP^3_{}}(4))=0,\, \forall i>0, i\neq 1
$$

$$
h^0(\mP^3,{\rm{Sym}}^{2}\Omega^1_{\mP^3_{}}(4))=20,\,\, h^i(\mP^3,{\rm{Sym}}^{2}\Omega^1_{\mP^3_{}}(4))=0,\, \forall i>0
$$
$$
h^i(\mP^3,{\rm{Sym}}^{2}\Omega^1_{\mP^3_{}})=0,\,\, i=0,1,2
$$
Then $$20=h^0(\mP^3,{\rm{Sym}}^{2}\Omega^1_{\mP^3_{}}(4))=
 h^0(X, {\rm{Sym}}^{2}\Omega^1_{\mP^3_{|X}}(4)),$$ 
$0=h^1(X, {\rm{Sym}}^{2}\Omega^1_{\mP^3_{|X}}(4))$ and 
$$
 h^0(X,{\rm{Sym}}^{3}\Omega^1_{\mP^3_{|X}}(8))=h^0(\mP^3,{\rm{Sym}}^{3}\Omega^1_{\mP^3_{}}(8))+h^1(\mP^3,{\rm{Sym}}^{3}\Omega^1_{\mP^3_{}}(4))=280+20.
$$
Then $h^0(X, {\rm{Sym}}^{3}\Omega^1_X(8))=280$. The claim follows since $h^0(S, {\rm{Sym}}^2 (\sQ_S) \otimes \sO_S (\sigma)\otimes_{\sO_S}\omega_S)=45+h^0(X, {\rm{Sym}}^{3}\Omega^1_X(8))=45+280=325$.

\end{proof}
\begin{thm}\label{domandareferee} $h^1(S,T_S)=20$.
\end{thm}
\begin{proof} By the cohomology of the sequence (\ref{tortatorta}), by Theorem \ref{formulae} and by Proposition \ref {calcolo}we have that
\begin{equation}\label{finefine}
h^0(S,\Omega^1_S\otimes_{\sO_S}\omega_S)=h^0(S, {\rm{Sym}}^2 (\sQ_S) \otimes \sO_S (\sigma)\otimes_{\sO_S}\omega_S)-h^0(S,\omega_S))=280.
\end{equation}

By Grothendiek-Riemann-Roch $\chi(T_S)=2\chi(\sO_S)+K_S^2-c_2(S)$ and by Theorem \ref{formulae} we obtain $\chi(T_S)=92+360-192=260$. By Serre duality and by equation (\ref{finefine}) $h^2(S,T_S)=h^0(S,\Omega^1_S\otimes_{\sO_S}\omega_S)=280$. Then $h^1(S,T_S)=20$ since $h^0(S,T_S)=0$.
\end{proof}

As a final check we can remark that the image of 
$$H^0(\nu)\colon H^0(S, {\rm{Sym}}^2(\sQ_S)\otimes_{\sO_S}\omega_S(1))\to H^0(S,\omega_S^{\otimes 2})$$ has dimension $386$ since a straightforward Chern class computation gives the following:
\begin{thm} It holds that $h^0(S, {\rm{Sym}}^2(\sQ_S)\otimes_{\sO_S}\omega_S(1))=666$.
\end{thm}

\begin{ackn} The authors are very grateful to an anonymous referee who read in detail a previous version of this paper, pointed out several inaccuracies and suggested many improvements. In particular, he urged us to carry out the study of deformations which is the object of the last section in the present version of the paper. The authors are pleased to thank also prof. P. Craighero for useful conversations during an early stage of the present work.
 
This research is partially supported by the project PRID-DIMA-Geometry and by the PRIN 2017 ``Geometric, Analytic and Algebraic Aspects of Arithmetic'' . \end{ackn}

\end{document}